\documentclass[twoside,12pt]{article}
\usepackage[centertags]{amsmath}
\usepackage{fancyhdr,lipsum}
\usepackage{amsfonts}
\usepackage{amssymb}
\usepackage{amsthm}
\usepackage{amsmath,mathrsfs}
\usepackage{amsmath,amscd}
\usepackage[matrix,arrow,curve,cmtip]{xy}
\title{\textbf{A Spectral ``Schematization" of Homotopy Types}}
\author{Renaud Gauthier  \\ \\}

\theoremstyle{definition}
\newtheorem{fctrchg}{Proposition}[subsection]
\newtheorem{contr}[fctrchg]{Proposition}
\newtheorem{discr}[fctrchg]{Proposition}
\newtheorem{prodct}[fctrchg]{Proposition}
\newtheorem{path}[fctrchg]{Proposition}
\newtheorem{colimit}[fctrchg]{Proposition}
\newtheorem{Frd}[fctrchg]{Theorem}
\newtheorem{basechg}[fctrchg]{Proposition}

\DeclareMathOperator*{\colim}{\text{colim}}

\newcommand{\beq}{\begin{equation}}
\newcommand{\eeq}{\end{equation}}

\newcommand{\hrarr}{\hookrightarrow}
\newcommand{\rarr}{\rightarrow}

\newcommand{\larr}{\leftarrow}
\newcommand{\Rarr}{\Rightarrow}

\newcommand{\xrarr}{\xrightarrow}

\newcommand{\cC}{\mathcal{C}}

\newcommand{\cM}{\mathcal{M}}

\newcommand{\cS}{\mathcal{S}}

\newcommand{\bR}{\mathbb{R}}
\newcommand{\bS}{\mathbb{S}}

\newcommand{\bZ}{\mathbb{Z}}

\newcommand{\Hom}{\text{Hom}}
\newcommand{\Ho}{\text{Ho}\,}

\newcommand{\Map}{\text{Map}}

\newcommand{\op}{\text{op}}

\newcommand{\Set}{\text{Set}}
\newcommand{\Spec}{\text{Spec\,}}
\newcommand{\Top}{\text{Top}}
\newcommand{\uHom}{\underline{\Hom}}

\newcommand{\AfX}{A \xrightarrow{f} X}

\newcommand{\BgY}{B \xrightarrow{g} Y}

\newcommand{\bfH}{\mathbf{H}}

\newcommand{\cDot}{\centerdot}

\newcommand{\Comm}{\text{Comm}}

\newcommand{\Dop}{\Delta^{\op}}

\newcommand{\fpinSpR}{f_*\pi_n(\Spec R)}

\newcommand{\Grp}{\text{Grp}}

\newcommand{\kCAlg}{k\text{-CAlg}}
\newcommand{\kCAlgpft}{\kCAlg_{pft}}
\newcommand{\kMod}{k\text{-Mod}}
\newcommand{\kAff}{k\text{-Aff}}
\newcommand{\kAffpft}{\kAff_{pft}}

\newcommand{\oT}{\otimes}

\newcommand{\pinSpRfX}{\pi_n(\Spec R \xrightarrow{f} X)}
\newcommand{\pinSpR}{\pi_n(\Spec R)}

\newcommand{\RuHom}{\bR \uHom}

\newcommand{\SpR}{\Spec R}
\newcommand{\Siginf}{\Sigma^{\infty}}

\newcommand{\SetD}{\Set_{\Delta}}

\newcommand{\SpS}{Sp^{\Sigma}}

\newcommand{\uA}{\underline{A}}

\newcommand{\SsS}{\cS_*^{\Sigma}}
\newcommand{\MapSsS}{\Map_{\SsS}}
\newcommand{\MapSpS}{\Map_{\SpS}}

\begin{document}
\maketitle

\begin{abstract}
	Toen has interpreted the schematization problem as originally imagined by Grothendieck (\cite{G2}) in such a way that solution(s) to this problem could be given. As he pointed out, there are many solutions available, and in \cite{T2}, \cite{T3}, \cite{T4} he gave two constructions solving this problem. What we do in the present work is reconsider what Grothendieck initially had in mind and develop a formalism that provides a concept of ``schematization" and corresponding homotopy groups, leading to a notion of ``schematized" homotopy types. In our view this can be realized if we generalize topological spaces to symmetric spectra, and the objects $\MapSsS(\bS, X)$ of $\Ho(\SetD)$ play the role of homotopy groups for spectra $X$, $\mathbb{S}$ the sphere spectrum, $\SsS$ the category of symmetric sequences.
\end{abstract}

\newpage

\section{Introduction}
It is important to note that the present work is in no way an addendum, nor does it supplant Toen's work on the schematization of homotopy types, which is very complete. Our aim is modest; we wanted to develop a formalism in which a schematization of homotopy types would make sense, along with an accompanying notion of homotopy groups from which such homotopy types would be defined, starting from very simple considerations. In doing so we arrived at a notion of homotopy groups of objects in a comma category, generalizing the classical homotopy groups of spaces, something that we found to be of interest to algebraic topologists. For this reason, after discussing a schematization of homotopy types proper, we study those generalized homotopy groups in a first part of the present work. In a second part, we systematize such an approach, and we find ourselves abandoning the base model category $\kCAlg$ in favor of the model category of symmetric spectra $\SpS$. Thus what we obtain at the onset should really be called a spectralization of homotopy types.\\

What Grothendieck started with is the observation in \cite{G2} that homotopy types are essentially discrete structures, so one may as well start our investigation with a free $\bZ$-module $M$ of finite type. It gives rise to a vector bundle over $S_0 = \Spec \bZ$ which he denoted by $W(M)$, whose $\bZ$-module of sections is $M$ itself. The functor $M \mapsto W(M)$ is fully faithful, the point being that this provides an embedding of homotopy types into a category of geometric objects such as bundles. He then considered any $\bZ$-module $M$, i.e., an abelian group. In that case our vector bundle should be defined as $W(M): k \mapsto M \otimes_{\bZ} k$ a functor on the category of commutative $\bZ$-algebras, still very close to being a vector bundle. That functor $M \mapsto W(M)$ is still faithful. This is this structure Grothendieck pulled back to spaces via torsors. What he considered then is an auxiliary category $U(k)$ of unipotent bundles over $k$, not quite schemes over $k$. They are functors $\kCAlg = (\kAff)^{\op} \rarr \Set$, i.e. presheaves on $\kAff$. If we denote by $\tau$ the fpqc topology on $\kAff$, then we consider the model category of sheaves on $\kAff$ that Toen denotes by $\kAff^{\sim, \tau}$, and we're looking for a sub-model category:
\beq
U(k) \subset \kAff^{\sim, \tau} \subset \widehat{\kAff} \nonumber
\eeq

\newpage
Grothendieck defined $U(k)$ as the union of an increasing sequence of subcategories $U_n(k)$. He defined $U_{n+1}(k)$ as follows: for any $k$-module $M$, the corresponding vector bundle is defined via $W(M)(k') = M \otimes_k k'$, and $X_{n+1} \in U_{n+1}(k)$ if and only if there is some $X_n \in U_n(k)$ and some $k$-module $M$ such that $X_{n+1}$ is isomorphic to a torsor over $X_n$ with group $W(M)$. In other terms, the initial embedding of a discrete structure $M$ into the category of schemes, or vector bundles as he saw them, this is pulled back to $U(k)$, thereby endowing the latter with the flavor of a schematic homotopy type. Later in his same work Grothendieck will observe that the modules $M$ have to be projective.\\

Recall that the sections functor on $W(M)$ gives us back $M$, which is supposed to represent a homotopy type. Thus we would like to have a sections functor in the present setting as well to extract the homotopy type information from objects $X$ of the category $U(k)$. It turns out we do have a sections functor $U(k) \rarr \Set: X \mapsto X(k)$ that induces a map:
\beq
\Hom(\Dop, U(k)) \rarr \Hom(\Dop, \Set) \nonumber
\eeq
Grothendieck initially argued that since the right hand side should model homotopy types, this makes simplicial objects of $U(k)$ into homotopy theoretic objects, which means we have an implicit notion of weak equivalence in $sU(k)$, hence a corresponding homotopy category. More specifically, he considered a subcategory $\cM(k)$ of $U(k)$ whose objects have homotopy groups endowed with a $k$-module structure, something that he will drop later, but nevertheless the idea remains: one should aim for a subcategory of $sU(k)$ in a sense to be precised. Then the notion of weak equivalence we have above should lead to a homotopy category $Hot_n(k)$ of $n$-connected schematic homotopy types over $k$, homotopy types because of a yet to be precised notion of weak equivalence, schematic because objects in $U(k)$ have the flavor of schemes insofar as they are torsors with groups $W(M)$ that initially were viewed as scheme-like objects. This is the picture we have to keep in mind. Much of the work Grothendieck did from there revolved around defining a proper notion of homotopy. The first hint that this preliminary picture may not work comes from the very definition of $\cM(k)$; how could one construct homotopy groups endowed with a non-trivial $k$-module structure? Looking elsewhere, Grothendieck pointed out that we do have a classical abelianization functor for homotopy types $Hot_n \rarr D_{\cDot}(Ab)$. \\

One should then have a commutative diagram:
\beq
\begin{CD}
Hot_n(k) @>>> D_{\cDot}(Ab_k)  \nonumber \\
@VVV @VVV \nonumber \\
Hot_n @>>> D_{\cDot}(Ab)
\end{CD}
\eeq
hence we need an abelianization, or linearization functor $Hot_n(k) \rarr D_{\cDot}(Ab_k)$, a localization of a functor $\cM_n(k) \rarr ch_{\cDot}(Ab_k)$. It is at this point that Eilenberg-Mac Lane spaces make their appearance, essentially as linearized, schematic homotopy types. Ultimately, Grothendieck was aiming for a $k$-linearization functor $L: U(k) \rarr Ab_k, X \simeq W(M) \mapsto L(X)$. Defining $L(X)$ necessitates taking completions, something he wanted to avoid, so he considered pointed homotopy types instead, something that is quite natural. Hence he considered a category $U(k)^{\centerdot}$ of pointed homotopy types instead of $U(k)$, with a corresponding $k$-linearization functor $L_{pt}$ that is well-behaved with regards to ring extensions. For $X_* \in sU(k)^{\centerdot}$ a simplicial object in $U(k)^{\centerdot}$, one defines its homology by $LH_{\centerdot}(X_*) = L_{pt}(X_*)$ and one says $X_* \rarr X_*'$ is a weak equivalence if the total homology functor $LH_{\centerdot}$ transforms this map into a quasi-isomorphism. Later Grothendieck defines $H_i(X_*):= \pi_i(L_{pt}(X_*)) = H_i(LH_{\centerdot}(X_*))$ something that he regarded as a much better way to define a notion of weak equivalence than the one just introduced. At some point later though Grothendieck introduced a Lie functor $Lie: U(k)^{\centerdot} \rarr Ab_k, X = W(M) \mapsto M$ with the usual homotopy groups $\pi_i(X_*) = \pi_i(X_*(k))$ satisfying $\pi_i(X_*) \simeq \pi_i(Lie(X_*))$ which he regarded as the natural definition of homotopy invariance of $X_*$, hence a notion of weak equivalence defined as producing quasi-isomorphisms for the corresponding Lie chain complexes. Even though this is yet another notion of weak equivalence which he considered as a cornerstone of a theory of schematic homotopy types, he will admit later that a notion of weak equivalence is not agreed upon. Thus the very definition of schematic homotopy type is not settled.\\

Regarding homotopy groups, them being $k$-modules will be dropped quite fast by Grothendieck, but one thing he will keep though is that higher homotopy groups $\pi_n$ for $n>1$ are endowed with an action of $\pi_1$, hinting at the fact that  homotopy groups themselves may possibly have to be upgraded to something more general. This is also apparent in the fact that since the beginning of his work on the schematization problem, he wanted to use Postnikov devissage to define schematic homotopy types, something he will find not to be too useful, yet not completely abandon, something he admitted later since it was constantly on his mind, and at times led him astray. This also points to the fact that homotopy groups carry an additional structure. It needs to be mentioned as well that Grothendieck at some point considered $\pi_1$ to be a group object in $U(k)$.\\

From this short synopsis one gets the flavor of what Grothendieck had in mind regarding schematic homotopy types, even though the very definition of weak equivalence to develop such a formalism remains elusive, and one cannot help but notice how homotopy groups themselves should probably be generalized.\\

One may summarize the above approach as working with schematic objects, containing some homotopy theoretic information, extracted in some sense with a sections functor, and whose intrinsic homotopic information is revealed via a notion of equivalence, defined as inducing quasi-isomorphisms for a well-chosen cohomology theory. What we will do instead of using a sections functor is work with covers of geometric objects $X$ by spectra of algebraic objects of finite type, thereby giving $X$ a sense in which it can be regarded as being schematic, and a homotopy theoretic object as well. The homotopy types proper will be defined exactly as Grothendieck defined them, that is as being dependent upon a notion of $\pi_*$-equivalence.\\

To talk about homotopy type one needs to have a model category in mind, or at the very least a category with a subclass of weak equivalences, and speaking of spaces, we are looking for a homotopy category of spaces. For a space $X$, speaking of schematic homotopy type over a ground ring $k$ means, from our perspective, putting $X$ and $k$ on a same footing for them to be comparable, hence in a first time we have to explain how $k$ could be seen as an object of a given homotopy category of spaces. This is fairly immediate: we consider the spectrum $\Spec k$ of a ring $k$, appropriately topologized. Now $X$ being given, and looking for a notion of homotopy type over $k$, we consider morphisms $\Spec R \rarr X$ in $\Top$ for $R \in \kCAlg$ as providing such a notion, a schematic cover of $X$ in a sense. Now recall that the modules $M$ Grothendieck started with are supposed to be homotopy types, and for this reason they are taken to be projective modules of finite type. Thus in the present context $R \in \kMod$ is projective of finite type. Let $\kCAlgpft = \Comm(\kMod_{pft})$ with obvious notations. Our cover can be formally represented as a functor:
\beq
	\Hom_{\Top}(\Spec(-), X): \kCAlgpft \rarr \Set \nonumber
\eeq
For the sake of extracting the homotopy theoretic information in such covers, we introduce a notion of homotopy group of coverings, by defining:
\beq
\pinSpRfX = \fpinSpR \nonumber
\eeq
This presupposes that we work with pointed spaces, and this will be implied throughout this work. Functorially this can be represented by a graded functor $\pi_*(\Spec(-) \rarr X)= \pi^k_{*,X}$ and weak equivalences are defined as $\pi_*$-equivalences, namely $X \sim X'$ if $\pi^k_{*,X} \cong \pi^k_{*,X'}$. In a first part of the present work, we will derive a few properties of such homotopy groups. In a second part, we will generalize those homotopy groups; we have $\pi^k_{n,X}(R) = \pi_n(\Spec R \xrarr{f} X) = f_* \pi_n(\Spec R)$, for $R \in \kCAlgpft$. Focusing on $\pi_n(\Spec R)$, if we want this group to be representable, we are looking for some $R_0 \in \kCAlgpft$ so that $\pi_n = \Hom_{\kCAlgpft}(-, R_0)_n = \Hom_{\kAffpft}(\Spec R_0, -)_n$, with $\kAffpft = \kCAlgpft^{\op}$. Written differently:
\beq
[S^n, \Spec R] =\pi_n(\Spec R) =  \Hom_{\kAffpft}(\Spec R_0, \Spec R)_n \nonumber \\
\eeq
which can be implemented by taking $R_0$ to be a graded object and the Hom object to be in some homotopy category. One can immediately think of spectra with $\Spec R_0$ a component of the sphere spectrum $\Siginf S^0 = \bS$, and $\Spec R$ being replaced by a component of an $E_{\infty}$-ring. Thus the functor $\pi_* = [S^*, -]$ should be replaced by $\prod_p \Map(S^p,-) = \Map(\bS,-)$ in a sense to be precised. Note that at this point there is a clear departure from working with objects $R \in \kCAlg$, since the most basic objects are symmetric spectra. Instead of topological spaces, we will consider simplicial sets $K$ and their symmetric suspension spectrum $\Siginf K = X$. Homotopy types will be represented by symmetric spectra of finite type, defined as being spectra for which its simplicial components are finite. We consider coverings of finite type $\{\bfH \rarr X \}$ for $\bfH \in \SpS_{ft}$, functorially represented as $\RuHom(-,X): \Ho(\SpS_{ft}) \rarr \Ho(\SpS)$, a perception of finite type, where $\SpS$ is endowed with the positive model structure of \cite{S}, $\uHom$ is the internal Hom of $\SpS$, and $\RuHom$ its derived Hom. For $f: \bfH \rarr X$, we have an induced map $\Ho(f_*): \Map(\bS, \bfH) \rarr \Map(\bS, X)$ of simplicial sets, and the collection of subobjects of the form $\Ho(f_*) \Map(\bS, \bfH)$ defines a functor $\Pi_X$. Then one defines the following weak equivalence: $X \sim X'$ in $\SpS$ if and only if $\Pi_X \cong \Pi_{X'}$. This gives us a notion of homotopy type, defined via symmetric spectra of finite type covering $X$ (resp. $X'$). Thus we have a spectral interpretation of the schematization of homotopy types.\\ 

Notations: we will denote by $\Hom_{\cC}(X,Y)$ the hom sets between objects $X$ and $Y$ of a category $\cC$, by $\Map_{\cC}(X,Y)$ the mapping space, a simplicial set, if $\cC$ is simplicial, and by $\uHom_{\cC}(X,Y)$ the internal Hom. If $\cC$ is a model category, we use the notation $\RuHom$ to denote the derived Hom. Typically we use the notation $\SetD$ for the category of simplicial sets, but for symmetric sequences we will use the compact notation $\SsS$ instead of the more cumbersome $\Set_{\Delta,*}^{\Sigma}$.

\section{Part I: Algebraic Topology}
\subsection{Generalized homotopy groups}
As pointed out in the introduction, for $X$ a given topological space, $k$ a commutative ring, to talk about the homotopy type of $X$ over $k$ one would want to put $X$ and $k$ on a same footing to be able to define invariants of one object relative to the other. Thus it is natural to consider morphisms $\Spec R \rarr X$ for $R \in \kCAlg$. In addition we ask that those $k$-algebras be projective of finite type. The image of that morphism tells us what is $X$ relative to $R$. The collection of all such maps provides an affine covering family $\{ \Spec R \rarr X \, | \, R \in \kCAlgpft \}$. This can further be formalized by considering the functor:
\beq
	\Hom_{\Top}(\Spec(-),X): \kCAlgpft \rarr \Set \nonumber
\eeq
which we defined as being the affine perception in Spectral Algebraic Geometry (\cite{RG3}).\\

A morphism $\Spec R \rarr X$ is regarded as an object of $\Top_{/X}$, and we generalize homotopy groups of spaces to homotopy groups of morphisms of spaces as follows. Consider the (local system) functor:
\begin{align}
\pi_n: \Top & \rarr \Grp^{\Top/\centerdot} \nonumber \\
X & \mapsto \pi_n(X):=\pi_{nX} \nonumber
\end{align}
where:
\begin{align}
\pi_{nX}: \Top_{/X} &\rarr \Grp \nonumber \\
(Y \xrightarrow{f} X) & \mapsto \pi_{nX}(Y \xrightarrow{f} X) = f_* \pi_n(Y) \nonumber
\end{align}

We will drop the subscript $X$ in what follows. Observe that the classical homotopy group is recovered as $\pi_n(X \xrightarrow{id} X) = id_* \pi_n(X) = \pi_n(X)$, where $X \xrightarrow{id} X$ is the terminal object in $\Top_{/X}$, so the classical picture can be seen as being embedded in this more general formalism. Observe that one may further generalize this by letting $Y$ itself in $\pi_n(Y \xrightarrow{f} X) = f_* \pi_n(Y)$ be an object in a comma category, with the present case above being recovered by letting $Y = Y$. Indeed:
\begin{align}
\pi_n(A \xrightarrow{f} B \xrightarrow{g} X) &= \pi_n(A \xrightarrow{g \circ f} X) \nonumber \\
&= (g \circ f)_* \pi_n(A) \nonumber \\
&= g_* \circ f_* \pi_n(A) \nonumber \\
&= g_* \pi_n(A \xrightarrow{f} B) \nonumber
\end{align}
providing a generalization of $\pi_n(Y \xrarr{g} X) = g_* \pi_n(Y)$ to the case where $Y$ is an object in a comma category.\\

We will be interested in the affine part of this functor, namely $\pi_{n,X} = \pi_n (\Spec (-) \rarr X)$, which we regard as a functor $\pi_n(- \rarr X)$ on $\kAffpft \times_{\kCAlgpft} X^{\kAffpft}$ whose objects are pairs $(\Spec R,  \Spec R \xrarr{f} X)$. However, $\pi_{n,X}$ is no bifunctor, it is really a functor in a comma category. A morphism $(\Spec R, f) \rarr (\Spec R' , f')$ in this category is a morphism $u: \Spec R \rarr \Spec R'$ that makes the following diagram commute:
\beq
\xymatrix{
	\Spec R \ar[dd]_u \ar[dr]^f \\
	&X \\
	\Spec R' \ar[ur]_{f'}
} \nonumber
\eeq
this induces the following commutative diagram of homotopy groups:
\beq
\xymatrix{
	\pi_n\Spec R \ar[dd]_{u_*} \ar[dr]^{f_*} \\
	& \pi_n X \\
	\pi_n \Spec R' \ar[ur]_{f'_*}
} \nonumber
\eeq
Thus we have a covariant functor:
\begin{align}
	\pi_{n,X}: \kAffpft \times_{\kCAlgpft} X^{\kAffpft} & \rarr \Grp \nonumber \\
	(\Spec R,  \Spec R \xrarr{f} X) & \mapsto f_* \pi_n \Spec R \nonumber
\end{align}

Thus defined, $\pi_{*,X}$ gives us a schematic probe of the homotopic character of $X$. Define two spaces $X$ and $X'$ to be equivalent if they are $\pi_*$-equivalent, that is $\pi_{*,X} \cong \pi_{*,X'}$. This provides us with a notion of weak equivalence, based on schematic objects, hence a notion of schematic homotopy type.\\

\subsection{Properties of homotopy groups}
We use mainly \cite{M} as a resource to develop our formalism. In terms of loop spaces, one can write:
\beq
\pi_n(\AfX) = f_* \pi_n(A) = f_* \pi_{n-1} \Omega A = \cdots = f_* \pi_0 (\Omega^n A) \nonumber
\eeq
For the purposes of having long exact sequences of homotopy groups, we consider, for $\AfX$, the following homotopy fiber:
\beq
Ff = A \times_f PX = \{(a, \xi) \; | \; f(a) = \xi(1) \} \subset A \times PX \nonumber
\eeq
with $PX = F(I,X)$, $I$ based at 0, $*$ fixed point of $X$, so $Ff$ consists of paths $\gamma: * \rarr f(A) \subset X$, i.e. $Ff = P(X;*, f(A))$, the space of paths in $X$ starting at the fixed point $*$ and ending in $f(A)$. We have a natural projection $\pi: Ff \rarr A, (x, \chi) \mapsto x$. We also have an inclusion $\iota: \Omega X \rarr Ff$ specified by $\iota(\xi) = (y, \xi)$, $y \in f(A)$. This gets us a fiber sequence generated by $f$ as given in \cite{M}:
\beq
\cdots \rarr \Omega^2A \xrightarrow{\Omega^2 f} \Omega^2 X \xrightarrow{-\Omega \iota} \Omega Ff \xrightarrow{ -\Omega \pi} \Omega A \xrightarrow{-\Omega f} \Omega X \xrightarrow{\iota} Ff \xrightarrow{ \pi} \AfX \nonumber
\eeq
where as defined in \cite{M}, $(-\Omega f)(\eta)(t) = (f \circ \eta)(1-t)$ for $\eta \in \Omega A$. From there we get a long exact sequence of pointed sets for any based space $Z$:
\beq
 \cdots \rarr [Z,\Omega Ff] \rarr [Z, \Omega A] \rarr [Z, \Omega X] \rarr [Z, Ff] \rarr [Z, A] \rarr [Z, X] \nonumber
\eeq
Then we let $Z = S^n$ and $[S^n, -] = \pi_n$ to arrive at:
\beq
\setlength{\unitlength}{0.5cm}
\begin{picture}(34,11)(0,0)
\thicklines
\put(-.2,10){$\cdots$}
\put(3.5,10){$\pi_n \Omega Ff$}
\put(9,10){$\pi_n \Omega A$}
\put(14,10){$\pi_n \Omega X$}
\put(19,10){$\pi_n Ff$}
\put(24,10){$\pi_n A$}
\put(28.5,10){$\pi_n X$}
\put(3.5,7){$\pi_{n+1} Ff$}
\put(9,7){$\pi_{n+1} A$}
\put(14,7){$\pi_{n+1}X $}
\put(17.5,7){$\pi_n P(X; *, f(A))$}
\put(1.5,4){$\pi_{n+1} P(X;*, f(A))$}
\put(18,4){$\pi_{n+1}(X, f(A))$}
\put(2.5,1){$\pi_{n+2}(X, f(A))$}
\put(1,10.2){\vector(1,0){2}}
\put(6.5,10.2){\vector(1,0){2}}
\put(11.5,10.2){\vector(1,0){2}}
\put(16.5,10.2){\vector(1,0){2}}
\put(21.5,10.2){\vector(1,0){2}}
\put(26,10.2){\vector(1,0){2}}
\put(4.5,8){\line(0,1){1.3}}
\put(4.8,8){\line(0,1){1.3}}

\put(10,8){\line(0,1){1.3}}
\put(10.3,8){\line(0,1){1.3}}

\put(15,8){\line(0,1){1.3}}
\put(15.3,8){\line(0,1){1.3}}

\put(20,8){\line(0,1){1.3}}
\put(20.3,8){\line(0,1){1.3}}

\put(4.5,5){\line(0,1){1.3}}
\put(4.8,5){\line(0,1){1.3}}

\put(20,5){\line(0,1){1.3}}
\put(20.3,5){\line(0,1){1.3}}

\put(4.5,2){\line(0,1){1.3}}
\put(4.8,2){\line(0,1){1.3}}
\end{picture} \nonumber
\eeq
in simplified form:
\beq
\cdots \rarr \pi_n A \xrightarrow{f_*} \pi_n X \rarr \pi_n(X, fA) \xrightarrow{\partial} \pi_{n-1}A \rarr \cdots \rarr \pi_0 A \rarr \pi_0 X \nonumber
\eeq
Coming back to the definition of homotopy groups of objects in comma categories, $\pi_n(A \xrightarrow{f} X) = f_* \pi_n A$ is the image of $\pi_n A \xrightarrow{f_*} \pi_n X$, so $\pi_n A \twoheadrightarrow \pi_n(\AfX)$, and because we have a long exact sequence, it follows that:
\beq
\pi_n A / \partial \pi_{n+1}(X, fA) \cong \pi_n (\AfX) \label{pidelpif}
\eeq
We would like now to answer the question as to whether a map $A \xrightarrow{\tau} B$ induces a map $\pi_n(\AfX) \rarr \pi_n(B \xrightarrow{g} X)$. In order to do so we use the standard definition for morphisms of objects of $\Top_{/X}$. For later purposes, we consider the following general case: for two objects $\AfX$ and $\BgY$, we define morphisms between those objects as being given by a commutative diagram $(\tau,h)$ of the form:
\beq
\xymatrix{
	A \ar[d]_f \ar[r]^{\tau} & B \ar[d]^g  \\
	X \ar[r]_h &Y 
} \label{CD1}
\eeq
\begin{fctrchg}
	Given $A \xrarr{f} X$ and $B \xrarr{g} Y$, a map $(\tau,h): (A,X) \rarr (B,Y)$ as in \eqref{CD1} induces a morphism $\pi_n(\AfX) \xrightarrow{(\tau,h)_*} \pi_n(\BgY)$. 
\end{fctrchg}
\begin{proof}
By commutativity of \eqref{CD1}, we have the following commutative diagram:
\beq
\begin{CD}
\pi_{n+1}(X, fA) @>\partial>> \pi_n A @>f_*>> \pi_n X \nonumber \\
@. @V\tau_*VV @VVh_*V \nonumber \\
\pi_{n+1}(Y,gB) @>>\partial> \pi_n B @>>g_*> \pi_n Y \nonumber
\end{CD}
\eeq
from which we have:
\begin{align}
\partial \pi_{n+1}(X,fA) = \text{Ker} f_* \subset \text{Ker } h_* \circ f_* &= \text{Ker} g_* \circ \tau_* \nonumber \\
&=\{u \in \pi_n A, \; \tau_*u \in \text{Ker} g_* \} \nonumber \\
&=\{ u \in \pi_n A, \; \tau_*u \in \partial \pi_{n+1}(Y,gB) \} \nonumber \\
&= \tau_*^{-1} \partial \pi_{n+1}(Y, gB) := H \nonumber
\end{align}
that is, $\partial \pi_{n+1}(X, fA) \subset H$. Now we have a map:
\beq
\pi_n A \rarr \pi_n B \twoheadrightarrow \pi_n B / \partial \pi_{n+1}(Y, gB) \nonumber \\
\eeq
and $H = \tau_*^{-1} \partial \pi_{n+1}(Y, gB) \vartriangleleft \pi_n A$, so there is a map:
\beq
\pi_n A /H \rarr \pi_n B / \partial \pi_{n+1}(Y, gB) \nonumber
\eeq
and since we have shown $\partial \pi_{n+1}(X, fA) \subset H$, it follows we also have a map:
\beq
\pi_n A/ \partial \pi_{n+1}(X, fA) \rarr \pi_n A/H \nonumber
\eeq
and by composition we get a map:
\beq
\pi_n A/ \partial \pi_{n+1}(X, fA) \rarr \pi_n B / \partial \pi_{n+1}(Y, gB) \nonumber
\eeq
that is, a morphism $(\tau, h)$ from $\AfX$ to $\BgY$ induces a map $\pi_n(\AfX) \rarr \pi_n(\BgY)$ by \eqref{pidelpif}. 
\end{proof}

\newpage
\begin{contr}
If $X$ is contractible, for any $\AfX$, then $\pi_n(\AfX) = 0$.
\end{contr}
\begin{proof}
$\pi_n(X) = 0$ for all $n \geq 0$, and $\pi_n(\AfX) = f_* \pi_n A \subset \pi_n X$ so that $\pi_n(\AfX) = 0$ for all $n\geq 0$.
\end{proof}
\begin{discr}
If $X$ is discrete, for any $\AfX$, $\pi_n(\AfX) = 0$.
\end{discr}
\begin{proof}
$X$ being discrete, $\pi_n X = 0$, the rest of the proof is identical to that of the previous lemma.
\end{proof}
\begin{prodct}
$\pi_n((\AfX) \times (\BgY)) \cong \pi_n(\AfX) \times \pi_n(\BgY)$.
\end{prodct}
\begin{proof}
It suffices to write:
\begin{align}
\pi_n((\AfX) \times (\BgY))&= \pi_n(A\times B \xrightarrow{f \times g} X \times Y) \nonumber\\
&=(f \times g)_* \pi_n(A \times B) \nonumber \\
& \cong (f \times g)_* \pi_n A \times \pi_n B \nonumber \\
&= f_* \times g_* \pi_n A \times \pi_n B \nonumber \\
&= f_* \pi_n A \times g_* \pi_n B \nonumber \\
&= \pi_n (\AfX) \times \pi_n (\BgY) \nonumber
\end{align}
\end{proof}
\begin{path}
$\pi_n(\AfX)$ does not depend on the basepoint.
\end{path}
\begin{proof}
For classical homotopy groups, this is proved in \cite{M} and is stated as follows: if $f:(X,A) \rarr (Y,B)$ is a map of pairs, $\alpha:I \rarr A$ a path from $a$ to $a'$ in $A$, then the following diagram commutes:
\beq
\begin{CD}
\pi_n(X,A,a) @>f_*>> \pi_n(Y,B,f(a))  \\
@V\cong VV @VV\cong V  \\
\pi_n(X,A,a') @>>f_*> \pi_n(Y,B,f(a')) \label{May}
\end{CD}
\eeq
In our setting this would read as follows: for a map $(\AfX) \xrightarrow{(\tau,h)} (\BgY)$, $\alpha:I \rarr A$ a path from $a$ to $a'$ in $A$, we prove the following diagram is commutative:

\beq
\begin{CD}
\pi_n(\AfX,a) @>(\tau,h)_*>> \pi_n(\BgY, \tau(a)) \nonumber \\
@V\cong VV @VV\cong V \nonumber \\
\pi_n(\AfX, a') @>>(\tau,h)_*> \pi_n(\BgY, \tau(a')) \nonumber
\end{CD}
\eeq
which can be rewritten in the following form:
\beq
\begin{CD}
f_* \pi_n(A,a) @>>> g_*\pi_n(B, \tau(a)) \nonumber \\
@VVV @VVV \nonumber \\
f_* \pi_n(A,a') @>>> g_* \pi_n(B, \tau(a')) \nonumber
\end{CD}
\eeq
This is what we prove. We start from:
\beq
\begin{CD}
A @>\tau>> B \nonumber \\
@VfVV @VVgV \nonumber \\
X @>> h> Y \nonumber
\end{CD}
\eeq
which produces the following commutative diagram:
\beq
\begin{CD}
\pi_n(A,a) @>\tau_*>> \pi_n(B, \tau a) \\
@Vf_*VV @VVg_*V  \\
\pi_n(X, fa) @>> h_*> \pi_n(Y, g\tau a)   \label{genMay}
\end{CD}
\eeq
It follows that along with \eqref{May}, we have:
\beq
	\xymatrix{
		f_*\pi_n(A,a) \ar[ddd]_{\cong} \ar@{.>}[rrr]^{h_*} &&& g_* \pi_n(B, \tau a) \ar[ddd]^{\cong} \\
		& \pi_n(A,a) \ar[ul]_{f_*} \ar[d]_{\cong} \ar[r]^{\tau_*} & \pi_n(B, \tau a) \ar[ur]^{g_*} \ar[d]^{\cong} \\
		& \pi_n(A,a') \ar[dl]^{f_*} \ar[r]_{\tau_*} & \pi_n(B, \tau a') \ar[dr]_{g_*} \\
		f_* \pi_n(A,a') \ar@{.>}[rrr]_{h_*} &&&g_* \pi_n(B, \tau a')
} \nonumber
\eeq
Now the center square commutes by \eqref{May}, the side squares commute by functoriality of $f_*$ and $g_*$, and the top and bottom squares commute by \eqref{genMay}, making the outside square commute.
\end{proof}

\subsection{Colimits}
We consider now the following behavior under colimits as covered in \cite{M}: if $X_i \hookrightarrow X_{i+1}$ are inclusions of spaces with colimit $X$, then the natural map $\colim_i \pi_n(X_i) \rarr \pi_n X$ is an isomorphism for each n. In our setting if $X_i \hookrightarrow X_{i+1}$ are inclusions with colimit $X$, so are $h_i: X_i \hookrightarrow X$. Consider:
\beq
\xymatrix{
	A_i \ar[d]_{f_i} \ar[r]^{\tau_i} & A \ar[d]^f \nonumber \\
	X_i \ar[r]_{h_i} & X
}
\eeq

We ask that $A_i \hookrightarrow A_{i+1}$ be inclusions as well so that $\tau_i$ itself is an inclusion. Hence we are looking at inclusions $(A_i \xrightarrow{f_i} X_i) \hookrightarrow (A_{i+1} \xrightarrow{f_{i+1}} X_{i+1})$, with $\colim_i(A_i \xrightarrow{f_i} X_i) = \AfX$, $A = \colim A_i$. The morphism $f: A \rarr X$ is defined as $f_i$ on $A_i$. Note that this can be generalized to other situations where $A_i$ and $X_i$ are objects of categories and are both ascending chains of objects by inclusion. We have:
\begin{colimit}
If $\AfX = \colim_n (A_n \xrightarrow{f_n} X_n)$, then the natural map:
\beq
\colim_n \pi_*(A_n \xrightarrow{f_n} X_n) \rarr \pi_*(\AfX) \nonumber
\eeq
is an isomorphism.
\end{colimit}
\begin{proof}
	Because $A_i \hrarr A_{i+1}$ (resp. $X_i \hrarr X_{i+1}$), we also have inclusions of groups $\pi_*(A_i) \hrarr \pi_*(A_{i+1})$ (resp. $\pi_*(X_i) \hrarr \pi_*(X_{i+1})$). We aim to prove that $\colim_n f_{n*}\pi_k(A_n) \cong f_* \pi_k(A)$ for all $k \geq 0$. Consider the ascending inclusions of groups and morphisms induced by the inclusions $A_i \hrarr A_{i+1}$ and $X_i \hrarr X_{i+1}$:
	\beq
	\xymatrix{
		\cdots \ar[r] & \pi_*(A_i) \ar[d]^{f_{i*}} \,\ar@{^{(}->}[r] & \pi_*(A_{i+1}) \ar[d]^{f_{i+1,*}} \,\ar@{^{(}->}[r] & \cdots \\
		\cdots \ar[r] &\pi_*(X_i) \,\ar@{^{(}->}[r] &\pi_*(X_{i+1})\, \ar@{^{(}->}[r] & \cdots
		} \nonumber
	\eeq
The colimit of such a diagram by the classical result as stated in \cite{M} is provided by $\pi_*(A) \xrarr{f_*} \pi_*(X)$. However we wish to show the colimit of the images of $\pi_*(A_i)$ in $\pi_*(X_i)$ under $f_{i*}$ are the images of $\pi_*(A)$ in $\pi_*(X)$ under $f_*$. Observe that those images in $\pi_*(X_i)$ form an ascending chain under inclusion, with their union as colimit, that is:
	\begin{align}
		\colim_n \pi_k(A_n \xrarr{f_n} X_n) &= \colim_n f_{n*} \pi_k(A_n) \nonumber \\
		&= \colim_n \text{Im} f_n|_{\pi_k(A_n)}   \nonumber \\
		&= \bigcup \text{Im}  f_n|_{\pi_k(A_n)}  \nonumber \\
		&= \text{Im}f|_{\pi_k(A)} \nonumber \\
		&= f_* \pi_k(A) = \pi_k(A \xrarr{f} X) \nonumber
	\end{align}
and this for all $k \geq 0$.
\end{proof}

\subsection{Generalized Freudenthal suspension theorem}
Recall from \cite{M} for example that for a based space $X$, we have the suspension homomorphism $\Sigma: \pi_n X \rarr \pi_{n+1} (\Sigma X)$ defined by $\Sigma f = f \wedge id: S^{n+1} \cong S^n \wedge S^1 \rarr X \wedge S^1 = \Sigma X$ giving rise to the following result: if $X$ is (n-1)-connected and non-degenerately based, then $\Sigma: \pi_q(X) \rarr \pi_{q+1} \Sigma X$ is a bijection for $q < 2n-1$ and a surjection for $q = 2n-1$. We prove:
\begin{Frd}
If $X$, $A$ are (n-1)-connected ($n\geq 1$) non-degenerately based spaces, we have:
\beq
\Sigma: \pi_q (\AfX) \rarr \pi_{q+1}(\Sigma A \xrightarrow{f \wedge id} \Sigma X) \nonumber
\eeq
is a bijection for $q < 2n-1$ and a surjection for $q = 2n-1$.
\end{Frd}

\newpage
\begin{proof}
We are asking whether we have a dotted map as shown below, satisfying the stated properties.
\beq
\setlength{\unitlength}{0.5cm}
\begin{picture}(11,11)(0,0)
\thicklines
\put(0,10){$\pi_q A$}
\put(0,7){$f_*$}
\put(0,5){$f_* \pi_q A$}
\put(3,10.4){$\Sigma$}
\put(6,10){$\pi_{q+1} \Sigma A$}
\put(8,7){$(f \wedge id)_*$}
\put(6,3){$\pi_{q+1} \Sigma X$}
\put(3,2){$\Sigma$}
\put(3,1){$\cong$}
\put(0,3){$\pi_q X$}
\put(6,5){$(f \wedge id)_* \pi_{q+1} \Sigma A$}
\put(2,10.2){\vector(1,0){3}}
\put(1,9){\vector(0,-1){3}}
\put(2,3.2){\vector(1,0){3}}
\put(7,9){\vector(0,-1){3}}
\multiput(2,5)(0.5,0){6}{\line(1,0){0.2}}
\put(5,5){\vector(1,0){0.1}}
\put(1,4){$\cap$}
\put(7,4){$\cap$}
\end{picture} \nonumber
\eeq
Consider $\gamma \in \pi_q A$. An element of $f_* \pi_q A \subset \pi_q X$ is of the form $f_* \gamma$, which is represented by $S^q \xrarr{\gamma} A \xrarr{f} fA$. Under $\Sigma$ on $\pi_q X$, it maps to $S^q \wedge S^1 \xrarr{\gamma \wedge id} A \wedge S^1 \xrarr{f \wedge S^1} fA \wedge S^1$, which reads $(f \wedge id)_* \Sigma \gamma$, and this last object is in $(f \wedge id)_* \pi_{q+1} \Sigma A$, thereby establishing the restriction of $\Sigma$ on $\pi_q X$ to $f_* \pi_q A$ does provide the desired dotted map. Because the bottom horizontal map is the suspension map on classical homotopy groups for which we have the Freudenthal suspension theorem, and by virtue of our commutative diagram, it follows that:
\beq
	\Sigma: f_* \pi_q A \rarr (f \wedge id)_* \pi_{q+1} \Sigma A \nonumber
\eeq
satisfies the same result as its classical counterpart, that is:
\beq
\Sigma: \pi_q (\AfX) \rarr \pi_{q+1}(\Sigma A \xrightarrow{f \wedge id} \Sigma X) \nonumber
\eeq
is a bijection for $q < 2n-1$ and a surjection for $q = 2n-1$.
\end{proof}

\subsection{Base change}
Finally we discuss base change: we use the following notations:
\begin{align}
	\pi_{nX}^k: \kAffpft \times_{\kCAlgpft} X^{\kAffpft} &\rarr \Grp \nonumber \\
	\{ \Spec k \larr \Spec R \rarr X \} &\mapsto \pi_n(\Spec R \rarr X) \nonumber
\end{align}
\begin{basechg}
$\pi^k_{nX}$ is functorial in $k$, that is a morphism $k \rarr k'$ of rings induces a natural transformation $\pi^k_{nX} \Rarr \pi^{k'}_{nX}$.
\end{basechg}
\begin{proof}
Suppose we have a morphism of base rings $b: k \rarr k'$. It induces $b^{\op}: \Spec k \larr \Spec k'$. Hence given $\{\Spec k' \larr \Spec R \rarr X \}$ we get:
\beq
\xymatrix{
	& \Spec R \ar[dl] \ar@{.>}[dd] \ar[dr] \\
	\Spec k' \ar[dr] && X \\
	& \Spec k
} \nonumber
\eeq
Thus $b^{\op}$ induces a map:
\beq
(b^{\op})_*: k'\text{-Aff}_{pft} \times_{k'\text{-CAlg}_{pft}}X^{k'\text{-CAlg}_{pft}} \rarr \kAffpft \times_{\kCAlgpft} X^{\kCAlgpft} \nonumber
\eeq
and finally by composition with $\pi_{nX}^k$ we get a map:
\beq
\pi_{nX}^k \circ (b^{\op})_*:  k'\text{-Aff}_{pft} \times_{k'\text{-CAlg}_{pft}}X^{k'\text{-CAlg}_{pft}} \rarr \Grp \nonumber
\eeq
that is, we have a map $((b^{\op})_*)^*: \pi_{nX}^k \rarr \pi_{nX}^{k'}$ induced by a base change morphism $b:k \rarr k'$.
\end{proof}

\section{Part II: Algebraic Geometry}
Homotopy groups of schematic homotopy types are endowed with some additional structure; initially they are asked to be $k$-modules, something that will be dropped later in \cite{G2}. $\pi_1$ is asked to operate on the $\pi_i$'s for $i >1$. $\pi_1$ itself is regarded as a group object in $U(k)$. As the next step towards defining such a notion of homotopy groups, we further generalize the functor $\pi_{*,X}$ introduced in the previous section.\\

We start by just considering $\pi_n: \kCAlg^{\op} \rarr \Set, R \mapsto \pinSpR$. If we want $\pi_n$ itself to be a sheaf on $\kCAlg$, we can make it representable by an object $R_0$ of $\kCAlg$, hence:
\beq
\pinSpR = \Hom_{\kCAlg}(R,R_0) = \Hom_{\kAff}(\Spec R_0, \SpR) \nonumber
\eeq
or in other terms:
\beq
[S^n,\SpR] = \Hom_{\kAff}(\Spec R_0, \SpR) \nonumber
\eeq
This points to the fact that the Hom on the right hand side is already taken in a homotopy category to be precised, and that probably $\Spec R_0$ should be regarded as a component of a semi-simplicial object, or even a spectrum. We are looking for such an object that would play the role of spheres in a homotopy category. We consider the category of symmetric spectra $\SpS$ as covered in \cite{HSS} and \cite{S}. We choose to represent spheres by the sphere spectrum $\Siginf S^0 = \bS$. Considering the sphere spectrum is suitable since it is the initial object in the $\infty$-category of ring spectra. Spaces are generalized to the setting of (topological) symmetric spectra via:
\beq
X \simeq S^0 \wedge X \xrightarrow{\Sigma} S^1 \wedge X \rarr \cdots \nonumber \\
\eeq
i.e. $X \mapsto \Siginf X$. Homotopically, we can work with simplicial sets instead. For $K$ a simplicial set, we therefore work with $\Siginf K$, the symmetric suspension spectrum of $K$. Therefore what used to be considered as $\Spec R$ in classical homotopy theory should be replaced by $A \in \SpS$.\\

Note the change of paradigm; we no longer work with $\kCAlgpft$  as our base model category, but $\SpS$. Thus strictly speaking we no longer work with schematic objects, but spectralizations thereof.\\

We will follow our classical study of schematic homotopy types as much as possible. One important note at this point: since we work with spectra, homotopy groups of such objects are typically defined using stable homotopy groups. We purposefully deviate from that norm since we are aiming for a far more rudimentary object that would encapsulate the notion of homotopy group. In order to define such an object, we first define symmetric sequences which are sequences of pointed simplicial sets with a basepoint preserving left action of $\Sigma_n$ on their $n$-th component. We consider $\Sigma = \coprod_{n \geq 0} \Sigma_n$ as a category with objects finite sets $\overline{n} = \{1,2,\cdots ,n\}$ with automorphisms as maps, and we denote by $\SsS$ the category of symmetric sequences. We have a forgetful functor:
\begin{align}
	\SpS & \rarr \SsS \nonumber \\
	A & \mapsto \uA \nonumber
\end{align}
with $A_n = (\uA)_n$. We then regard $\Hom_{\cS_*^{\Sigma_n}}(S^n,A_n)$ as the object corresponding to $\pi_n(A_n)$ in this setting. Taking the product of all such sets we get:
\beq
\prod_p \Hom_{\cS_*^{\Sigma_p}}(S^p,A_p) = \Hom_{\SsS}(\underline{\bS},\uA) \nonumber
\eeq
for $A \in \SpS$ as shown in \cite{HSS}. Turning those into simplicial sets by using mapping spaces $\Map(X,Y) = \Hom(X \wedge \Delta[-]_+,Y)$, with $\Delta[-]: \Delta \rarr \SetD$, $\Delta[-]_+ = \Delta[-] \coprod \Delta[0]$, we still have $\MapSsS(\underline{\bS},\uA) = \prod_p \Map_{\Sigma_p}(S^p,A_p)$. To lift this up to the level of spectra, we use the tensor product of symmetric sequences as introduced in \cite{HSS}, and we have:
\beq
\pi_*(A) := \MapSsS(\underline{\bS},\uA) \cong \MapSpS(\bS \otimes \underline{\bS},A)  \nonumber
\eeq
for $A \in \SpS$, where we used the fact that $\bS \otimes - :\SsS \rarr \SpS$ is the left adjoint of the forgetful functor $\SpS \rarr \SsS$. \\

The role played earlier by schematic coverings of the form $\{\Spec R \rarr X \}$ for $R \in \kCAlgpft$, or equivalently by the affine perception functor $\Hom(\Spec(-), X)$, is now played by the functor $\RuHom(-,A): \Ho(\SpS_{ft}) \rarr \Ho(\SpS)$, where $\SpS_{ft}$ corresponds to the category of symmetric spectra of finite type (simply defined as spectra with finite simplicial components), which we regard as homotopy type objects, and we put the stable model structure of \cite{S} on $\SpS$. Recall that earlier if we had a morphism $\Spec R \xrarr{f} X$, this gave rise to a subgroup $f_* \pi_n \Spec R$, the collection of whose objects we turned into a functor. In the present context, for $\bfH \in \SpS_{ft}$ and a given morphism $f: \bfH \rarr A$, we have an induced map of simplicial sets in $\Ho(\SetD)$:
\beq
\pi_*\bfH = \MapSpS(\bS \oT \bS, \bfH) \xrarr{\Ho(f_*)} \MapSpS(\bS \oT \bS, A) = \pi_* A \nonumber
\eeq

Then we define the homotopy group of the cover $\bfH \rarr A$ to be given by $\Ho(f_*) \Map(\bS \oT \bS, \bfH)$. Functorially, the spectral generalization of the homotopy group functor $\pi_{*,X}$ for $X$ a topological space is given by the functor:
\begin{align}
	\Pi_A: \SpS \times_{\SpS_{ft}} A^{\SpS} & \rarr \Ho(\SetD) \nonumber \\
	(\bfH, \bfH \xrarr{f} A) & \mapsto \Ho(f_*) \MapSpS(\bS \oT \bS, \bfH) \nonumber
\end{align}
Thus defined, we probe $A$ with homotopy types $\bfH$. Thus using this notion of homotopy group to define $\pi_*$-equivalences will yield spectral homotopy types: we define the following equivalence in $\SpS$: $A \sim A'$ if and only if $\Pi_A\cong \Pi_{A'}$. This constitutes a spectralization of the schematic homotopy type we constructed in the first part of this work.\\

\bigskip
\footnotesize
\noindent
\par \nopagebreak \noindent \textit{e-mail address}: \texttt{rg.mathematics@gmail.com}.


\begin{thebibliography}{50}
	\bibitem[RG3]{RG3} R. Gauthier, \textit{A Dual Representation in Spectral Algebraic Geometry}, arXiv:2006.15687[math.AG].
	\bibitem[G2]{G2} A. Grothendieck, \textit{Pursuing Stacks}, unpublished manuscript.
	\bibitem[HSS]{HSS} M. Hovey, B. Shipley, J. Smith, \textit{Symmetric Spectra}, J. Amer. Math. Soc. \textbf{13} (2000), no.1, 149-208.
	\bibitem[M]{M} P. May, \textit{A Concise course in Algebraic Topology}, (Chicago Lectures in Mathematics), University of Chicago Press, Chicago, IL, 1999.
	\bibitem[S]{S} B. Shipley, \textit{A Convenient Model Category for Commutative Ring Spectra}, (Homotopy Theory: Relations with Algebraic Geometry, Group Cohomology and Algebraic K-Theory), 473-483, Contemp. Math., \textbf{346}, Amer. Math. Soc., Providence, RI, 2004.
	\bibitem[T2]{T2} B.Toen, \textit{Homotopical and Higher Categorical Structures in Algebraic Geometry}, habilitation thesis, arxiv:math/0312262 [math.AG].
	\bibitem[T3]{T3} B.Toen, \textit{Champs Affines}, Selecta Math. (N.S) \textbf{12} (2006), no.1, 39-135.
	\bibitem[T4]{T4} B.Toen, \textit{Dualite de Tannaka Superieure: Structures Monoidales}, preprint available at www.math.univ-toulouse.fr/~btoen/tan.pdf.
	\bibitem[TV5]{TV5} B.Toen, G.Vezzosi, \textit{Brave New Algebraic Geometry and Global Derived Moduli Spaces of Ring SPectra}, (Elliptic Cohomology) 325-359, London Math. Soc. Lecture Note Ser., \textbf{342}, Cambridge Univ.Press, Cambridge, 2007.
	\bibitem[TV2]{TV2} B.Toen, G. Vezzosi, \textit{Algebraic Geometry over Model Categories: A general approach to derived algebraic geometry}, arXiv:math.AG/0110109.
	\bibitem[TV]{TV} B.Toen, G.Vezzosi, \textit{Homotopical Algebraic Geometry I: Topos Theory}, Adv. Math. \textbf{193} (2005), no.2, 257-372.
	\bibitem[TV4]{TV4} B.Toen, G.Vezzosi, \textit{Homotopical Algebraic Geometry II: Geometric Stacks and Applications}, Mem. Amer. Math. Soc. \textbf{193} (2008), no.902, x+224p.
	\bibitem[Vo]{Vo} R. Vogt, \textit{Boardman's Stable Homotopy Category}, Aarhus Univ. Lect. Notes Vol. \textbf{21}, Aarhus University, 1970.
\end{thebibliography}
\end{document}